\documentclass[12pt,reqno]{amsart}
\usepackage{libertine}
\usepackage[T1]{fontenc}
\usepackage[
	noamssymbols,
	]{newtxmath}
\usepackage[
  hmarginratio={1:1},     
  vmarginratio={1:1},     
  textwidth=400pt,        
  heightrounded,          
]{geometry}
\usepackage{lipsum}
\usepackage{amsfonts}
\usepackage[active]{srcltx}
\usepackage{amsmath,amssymb,amsthm,bm}
\usepackage{mathtools}
\usepackage{shuffle}
\usepackage{bm} 
\usepackage{enumitem}
\usepackage{wasysym}
\usepackage{mathscinet}
\usepackage{verbatim}
\usepackage{graphicx}
\usepackage[
pagebackref,
bookmarks=true,         
bookmarksnumbered=true, 
colorlinks=true,
pdfstartview=FitV,
linkcolor=blue, citecolor=blue, urlcolor=blue]{hyperref}
\usepackage[backrefs,alphabetic,abbrev,msc-links]{amsrefs}
\usepackage{microtype}
\theoremstyle{plain}
\newtheorem{theorem}{Theorem}[section]

\newtheorem{lemma}[theorem]{Lemma}

\newtheorem{proposition}[theorem]{Proposition}

\theoremstyle{remark}

\theoremstyle{definition}
\newtheorem{remark}[theorem]{Remark}

\def\RR{\mathbb{R}}
\def\EE{\mathbb{E}}

\def\PP{\mathbb{P}}
\def\NN{\mathbb{N}}

\def\cee{{\mathcal E}}

\def\cdd{{\mathcal D}}

\def\1{\bm 1}
\def\lt{\left}
\def\rt{\right}
\def\e{\mathrm{e}}
\let\Section=\section
\def\section{\setcounter{equation}{0}\Section}

\title[Stable Dawson-Watanabe processes in supercritical regimes] {Long-time asymptotic of stable Dawson-Watanabe processes in supercritical regimes}

\date{May 30, 2017}
\author[K. L\^e]{Khoa L\^e}
\email{knle@ualberta.ca}
\address{Department of Mathematical and Statistical Sciences\\
University of Alberta\\
 Edmonton, AB  T6G 2G1 Canada
}
\thanks{The author thanks PIMS for its support through the Postdoctoral Training Centre in Stochastics.}
\begin{document}
\begin{abstract} Let $W=(W_t)_{t\ge0}$ be a supercritical $\alpha$-stable Dawson-Watanabe process (with $\alpha\in(0,2]$) and $f$ be a test function in the domain of $-(-\Delta)^{\frac \alpha2}$ satisfying some integrability condition. Assuming the initial measure $W_0$ has a finite positive moment, we determine the long-time asymptotic of all orders of $W_t(f)$. In particular, it is shown that the local behavior of $W_t$ in long-time is completely determined by the asymptotic of the total mass $W_t(1)$, a global characteristic.
\end{abstract}
\maketitle
\section{Introduction}
	Let $W=(W_t,t\ge0)$ be a Dawson-Watanabe process starting from a finite measure $m$ with motion generator $-(-\Delta)^{\frac \alpha2} $ on $\RR^d$ ($\alpha\in(0,2]$) and linear growth $\beta>0$. More precisely, $W$ is a measure-valued Markov process such that the process
	\begin{equation}\label{def:M}
		t\to M_t(f) :=W_{t}\left(f\right)-m \left(f\right)-\int _{0}^{t}W_{s}\left((-(-\Delta)^{\frac \alpha2}+\beta)f\right){d}s
	\end{equation}
	is a martingale with quadratic variation 
	$ \langle M^W\left(f\right)\rangle_{t}=\int _{0}^{t}W_{s}\left( f^{2}\right){d}s$ 
	for all $ f\in C^{2}_b\left(\mathbb R^{d}\right)$.
	The law of $W$ is denoted by $\PP_m$.
	Throughout the paper, we assume that the initial measure $m$ has a finite positive moment, that is
	\begin{equation}\label{con:m}
		\int_{\RR^d}|x|^am(dx)<\infty \quad\textrm{for some}\quad a>0\,.
	\end{equation}
	The case when $\beta>0$ is known as supercritical branching regime. 
	The cases $\beta<0$ and $\beta=0$ are known respectively as subcritical and critical branching regimes which, however, are not considered in the current article. For a fixed test function $f$ with sufficient regularity and integrability, we investigate the long-time asymptotic of $W_t(f)$ in supercritical branching regimes. 
	To state the main result precisely, we prepare some notation. For each multi-index $k=(k_1,\dots,k_d)\in\NN^d$ and $x=(x_1,x_2,\dots,x_d)\in\RR^d$, we denote
	\begin{gather*}
		|k|=k_1+k_2+\cdots+k_d\,,\quad k!=k_1!k_2!\cdots k_d!\,,
		\quad x^k=x_1^{k_1}x_2^{k_2}\cdots x_d^{k_d}\,,
	\end{gather*}
	and define the constant $\vartheta^k_{d,\alpha}$ and the $\sigma$-finite signed measure $\lambda_d^k$ on $\RR^d$ respectively by
	\begin{equation}\label{def:constTheta}
		\vartheta^k_{d,\alpha}=\frac1{(2 \pi)^d}\int_{\RR^d}e^{-|\theta|^\alpha}\theta^k d \theta
		\quad\textrm{and}\quad 
		\lambda_d^k(dy)=\frac1{k!}y^kdy \,.
	\end{equation}
	Obviously $\lambda^0_d$ is the Lebesgue measure $\lambda_d$. In supercritical regimes ($\beta>0$), it is well-known that the limit $\lim_{t\to\infty}e^{-\beta t}W_t(1)$ exists almost surely and is a well-defined random variable. We denote $\widetilde W_\infty(1)=\lim_{t\to\infty}e^{-\beta t}W_t(1)$.
 
	\begin{theorem} \label{thm:stbSLLn}	
		Let $N$ be a non-negative integer and $f$ be a function in $\cdd(-(-\Delta)^{\frac \alpha2} )$ satisfying 
		\begin{equation}\label{con:fxN}
			\int_{\RR^d}|f(x)||x|^{N}dx<\infty\,.
		\end{equation} 
		Then, with $\PP_m$-probability one,
		\begin{equation}\label{stb1}
			\lim_{t\to\infty}t^{\frac {N+d} \alpha}\lt|W_t(f)-\widetilde W_\infty(1)\sum_{\substack{k\in\NN^d:|k|\le N\\|k|\textrm{ is even}}} {(-1)^{\frac{|k|}2}t^{-\frac{d+|k|}\alpha}}\vartheta^k_{d,\alpha}\lambda^k_d(f) \rt|=0\,.
		\end{equation}
		Written another way, we have 
		\begin{equation}\label{stb2}
		 	t^{\frac d \alpha} W_t(f)=\widetilde W_\infty(1)\sum_{\substack{k\in\NN^d:|k|\le N\\|k|\textrm{ is even}}} (-1)^{\frac{|k|}2}t^{-\frac{|k|}\alpha} \vartheta^k_{d,\alpha} \lambda^k_d(f)+o(t^{-\frac{N}\alpha}) \,,
		 \end{equation}
		where $\lim_{t\to\infty}t^{\frac N \alpha} o(t^{-\frac{N}\alpha})=0$ almost surely.
	\end{theorem}
	Herein, $\cdd(-(-\Delta)^{\frac \alpha2})$ denotes the domain of the weak generator of the $\alpha$-stable process (see the following section for a precise definition).
	Theorem \ref{thm:stbSLLn} extends results of Kouritzin and Ren \cite{MR3131303} in which the first order asymptotic ($N=0$) was identified. 
	For a heuristic explanation of long-time limits of supercritical superprocesses and their connection with strong laws of large numbers, we refer to \cite{kouritzin2016laws}*{Section 2} and \cite{kl2017FV}*{Subsection 2.2}.

	The higher order asymptotic expansions \eqref{stb1} and \eqref{stb2} are obtained by combining the method initiated by Asmussen and Hering \cite{MR0420889} and an asymptotic expansion of the $\alpha$-stable semigroup (see Proposition \ref{prop:semiexpan} below).
	When $-(-\Delta)^{\frac \alpha2}$ is replaced by the generator of an Ornstein-Uhlenbeck process, similar results have been obtained by Adamczak and Mi\l o\'s \cite{MR3339862}. 
	In such case, because of the exponential rates in the expansion of the Ornstein-Uhlenbeck semigroup, convergences in distribution are expected in the asymptotic of high orders (which are called central limit theorems). 
	On the other hand, the rates in the asymptotic expansion of the $\alpha$-stable semigroup are those of polynomials (see \eqref{stb:Pexpansion} below) and are negligible under the exponential growing expected total mass $W_t(1)$, which leads to almost sure limits in the asymptotic of all high orders.
	In view of Theorem \ref{thm:stbSLLn}, it is interesting to observe that the local behavior of $W_t$ in long time is completely determined by the asymptotic of the total mass $W_t(1)$, which is a global characteristic. 

	We conclude the introduction with an outline of the article. Section \ref{sec:martingale_formuation} reviews the martingale formulations of Dawson-Watanabe processes. In Section \ref{sec:characteristic_martingales}, we investigate the long-time asymptotic of $W_t$ against some special test functions. The proof of Theorem \ref{thm:stbSLLn} is presented in Section \ref{sec:proof_of_the_main_result}.

\section{Martingale formulations} 
\label{sec:martingale_formuation}
	We use $\nu(f)$ and $\langle f, \nu \rangle$ to denote $\int_{\RR^d} f d \nu$ for a measure $\nu$ and an integrable function $f$.
	Let $T_t$ be the semigroup corresponding to a symmetric $\alpha$-stable process acting on $b\cee(\RR^d)$, the space of bounded Borel measurable functions on $\RR^d$. In particular, for every $f\in b\cee(\RR^d)$,
	\begin{equation}\label{eqn:StbPt}
		T_t f(x)=\int_{\RR^d}p_t(x-y)f(y)dy\,,
	\end{equation}
	where $p_t$ is the probability transition kernel
	\begin{equation}\label{eqn:density}
		p_t(x)=\frac{1}{(2 \pi)^d}\int_{\RR^d}e^{i x\cdot \theta}e^{-t|\theta|^\alpha} d \theta\,.
	\end{equation}
	Let $\hat f$ denote the Fourier transform of $f$ with the normalization $\hat f(\theta)=\int_{\RR^d}e^{-i \theta\cdot x}f(x)dx$. 
	Using Fourier transform, $P_t f$ takes a simpler form
	\begin{equation}\label{eqn:StbPhat}
		T_t f(x)=\frac1{(2 \pi)^d}\int_{\RR^d}e^{ix\cdot \theta-t|\theta|^\alpha}\hat f(\theta)d \theta\,.
	\end{equation}
	The weak domain of $-(-\Delta)^{\frac \alpha2}$, denoted by $\cdd(-(-\Delta)^{\frac \alpha2})$, is the collection of all functions $f$ in $b\cee(\RR^d)$ such that the limit
	\begin{equation*}
		\lim_{t\downarrow0}\frac1t(T_tf-f)
	\end{equation*}
	exists pointwise and is a bounded measurable function on $\RR^d$. If $f$ belongs to $\cdd(-(-\Delta)^{\frac \alpha2})$, we denote the above limit by $-(-\Delta)^{\frac \alpha2}f$.

	We define $\widetilde W_t=e^{-\beta t}W_t$ and $\widetilde M_t=\int_0^te^{-\beta s}dM_s$. It follows from \eqref{def:M} and It\^o formula that for every $f\in\cdd(-(-\Delta)^{\frac \alpha2}) $
	\begin{equation}\label{eqn:tildeWM}
		\widetilde W_t(f)=m(f)+\int_0^t\widetilde W_s(-(-\Delta)^{\frac \alpha2}f)ds+\widetilde M_t(f)\,.
	\end{equation}
	Note that $t\to\widetilde M_t(f)$ is a martingale with quadratic variation
	\begin{equation*}
		\langle\widetilde M(f)\rangle_t=\int_0^t e^{-2 \beta s}W_s(f^2)ds=\int_0^t e^{-\beta s}\widetilde W_s(f^2)ds\,.
	\end{equation*}
	The measure-valued process $(\widetilde M_t)_{t\ge0} $ can be considered as a worthy martingale measure (cf. \cite{MR876085}) with dominating measure $K(dx,dy,ds)=\widetilde W_s( dx\times dy)ds$. In this sense, for every deterministic function $f:\RR_+\times\RR^d\to\RR$ and $t>0$ satisfying
	\begin{equation*}
		\EE\int_0^t\int_{\RR^d}e^{-\beta s}\widetilde W_s( f^2_s)ds<\infty  \,,	
	\end{equation*}
	one can define the stochastic integration $\widetilde M_t(f):=\int_0^t\int_{\RR^d}f(s,y)d\widetilde M(s,y)$ such that $(\widetilde M_t(f))_{t\ge0} $ is a martingale with quadratic variation
	\begin{equation}\label{eqn:quadM}
		\langle \widetilde M(f)\rangle_t=\int_0^t e^{-\beta s} \widetilde W_s( f^2_s)ds\,.
	\end{equation}
	It follows that (see \cite{MR1915445}*{pg. 167}) for every $f\in b\cee(\RR^d)$
	\begin{equation}\label{X.Green}
		\widetilde W_t(f)=m(P_t f)+\int_0^t\int_{\RR^d}T_{t-r}f(y)d\widetilde M(r,y)\,,
	\end{equation}
	which is called the Green function representation. From \eqref{X.Green}, one derives the following two important identities
	\begin{equation}\label{eqn:WW}
		\widetilde W_t(f)-\widetilde W_s(T_{t-s}f)=\int_s^t\int_{\RR^d}T_{t-r}f(y)d\widetilde M(r,y)
	\end{equation}
	and
	\begin{equation}\label{eqn:1stXmoment}
		\PP_m\widetilde  W_t(f)=m(T_tf)\,,
	\end{equation}
	which are valid for all $0\le s\le t$ and $f\in b\cee(\RR^d)$.
	The following estimate is intrinsic to supercritical regimes and plays a central role in our approach.
	\begin{lemma}\label{lem:WW}
		For every $f\in b\cee(\RR^d)$ and $t\ge s\ge 0$, we have
		\begin{equation*}
			\PP_m\lt[\lt(\widetilde W_t(f)-\widetilde W_s(T_{t-s}f)\rt)^2 \rt]\lesssim m(1)\|f\|^2_{\infty}  e^{-\beta s}\,.
		\end{equation*}
	\end{lemma}
	\begin{proof}
		From \eqref{eqn:WW}, \eqref{eqn:quadM} and \eqref{eqn:1stXmoment}
		\begin{align*}
			\PP_m\lt[\lt(\widetilde W_t(f)-\widetilde W_s(T_{t-s}f)\rt)^2 \rt]
			&= \PP_m\int_s^t \widetilde W_r((T_{t-r}f)^2)e^{-\beta r} dr
			\\&= \int_s^t \langle T_r(T_{t-r}f)^2,m \rangle e^{-\beta r} dr\,.
		\end{align*}
		We observe that $T_r(T_{t-r}f)^2\le T_rT_{t-r}(f^2)=T_t(f^2)$
		by Jensen's inequality and $\|T_tf^2\|_\infty\le\|f\|^2_{\infty}$.
		Hence, $\langle T_r(T_{t-r}f)^2,m \rangle\le  \|T_t(f^2)\|_\infty m(1)\le m(1)\|f\|_{\infty}^2$. It follows that
		\begin{equation*}
			\PP_m\lt[\lt(\widetilde W_t(f)-\widetilde W_s(T_{t-s}f)\rt)^2 \rt]
			\le m(1)\|f\|_{\infty}^2 \int_s^t e^{-\beta r}dr\,,
		\end{equation*}
		which yields the result.
	\end{proof}
\section{Characteristic martingales} 
\label{sec:characteristic_martingales}
	For every $x,\theta\in\RR^d$, we denote $\e_\theta(x)=e^{i \theta\cdot x}$, $\cos_\theta(x)=\cos(\theta\cdot x)$ and $\sin_\theta(x)=\sin(\theta\cdot x)$ and recall the assumption \eqref{con:m} on $m$ and the definition of $\vartheta^k_{d,\alpha}$ in \eqref{def:constTheta}. We investigate the long-time asymptotic of $\widetilde M_t(\e_\theta)$. 

\begin{lemma}\label{lem:supM} For every $\theta\in\RR^d$,
	\begin{align*}
			\PP_m\lt[\sup_{t\ge0}|\widetilde M_t(\e_\theta)-\widetilde M_t(1)|^2\rt]
			\lesssim |\theta|^{2\wedge a}+ |\theta|^\alpha\,.
	\end{align*}
\end{lemma}
\begin{proof}
	For each $\theta\in\RR^d$, $(\widetilde M_t(\e_{\theta}))_{t\ge0}$
	is a complex valued martingale whose real and imaginary parts have quadratic variations satisfying
	\begin{align}\label{eqn:qudM}
	\langle\Re\widetilde M(\e_\theta)\rangle_{t} =
	 \int _{0}^{t}\widetilde W_r\left(\cos^{2}_{\theta }\right)e^{-\beta r} dr
	\enskip\textrm{and}\enskip
	\langle\Im \widetilde W(\e_\theta)\rangle_{t}
	 =  \int _{0}^{t}\widetilde W_r\left(\sin^{2}_{\theta }\right)e^{-\beta r}dr\,.
	\end{align}
	Together with martingale maximal inequality, we see that
	\begin{align*}
		\PP_m\lt[\sup_{t\ge0}|\widetilde M_t(\e_\theta)-\widetilde M_t(1)|^2\rt]
		\lesssim\PP_m\int_0^\infty \widetilde W_r\lt((\cos_\theta-1)^2+\sin_\theta^2 \rt)e^{-\beta r}dr\,.
	\end{align*}
	Hence, using the elementary identity $(\cos_\theta-1)^2+\sin_\theta^2=4\sin^2_{\theta/2}$ and \eqref{eqn:1stXmoment}, we obtain
	\begin{align*}
		\PP_m\lt[\sup_{t\ge0}|\widetilde M_t(\e_\theta)-\widetilde M_t(1)|^2\rt]
		\le\int_0^\infty \langle T_r(\sin^2_{\theta/2}),m \rangle e^{-\beta r}dr\,.
	\end{align*}
	Note that for every $x\in\RR^d$
	\begin{align*}
		2T_r\sin^2_{\theta/2}(x)&=1-\cos_{\theta}(x)e^{-r|\theta/2|^\alpha}=(1-\cos_{\theta}(x))e^{-r|\theta/2|^\alpha}+1-e^{-r|\theta/2|^\alpha}
		\\&\lesssim (1\wedge|\theta||x|)^2+r|\theta|^\alpha\,.
	\end{align*}
	These estimates and \eqref{con:m} implies the result.
\end{proof}
\begin{lemma}\label{lem:Mlim}
	$\widetilde M_t(\e_\theta) $ converges almost surely and in the mean-square sense to limit $\widetilde M_\infty(\e_\theta)$ for each $\theta\in\RR^d$. In addition, the following relation holds
	\begin{equation}\label{eqn:MW}
		\widetilde W_\infty(1)=m(1)+\widetilde M_\infty(1)\,.
	\end{equation}
\end{lemma}
\begin{proof}
	Using \eqref{eqn:qudM} we have
	\begin{align*}
		\PP_m|\widetilde M_t(\e_\theta)|^2=\int_0^t \PP_m \widetilde W_r(1) e^{- \beta r}dr\,,
	\end{align*}
	which together with \eqref{eqn:1stXmoment} implies $\sup_{t\ge0}\PP_m|\widetilde M_t(\e_\theta)|^2<\infty$. Hence, by the martingale convergence theorem, $\lim_{t\to\infty}\widetilde M_t(\e_\theta)$ exists almost surely and in mean-square sense for each $\theta\in\RR^d$. The relation \eqref{eqn:MW} follows from here (by setting $\theta=0$) and the relation \eqref{X.Green} with $f\equiv 1$.
\end{proof}
\begin{proposition}	\label{prop:Wp}
	Let $\rho(t)=t^\kappa$ for some $\kappa\in(0,1)$.
	With $\PP_m$-probability one, we have for every $k\in\NN^d$ that
	\begin{equation}\label{stb:pk}
			\lim_{t\to\infty}t^{\frac{d+|k|}\alpha} \widetilde W_{\rho(t)}(\partial^kp_{t- \rho(t)})=
			\lt\{
			\begin{array}{ll}
				0& \quad\mbox{if } |k|\mbox{ is odd}\\
				(-1)^{\frac {|k|}2} \vartheta^k_{d,\alpha}\widetilde W_\infty(1) & \quad\mbox{if } |k|\mbox{ is even}\,.
			\end{array}
			\rt.
	\end{equation}
\end{proposition}
	\begin{proof}
		It suffices to restrict on the event $\widetilde W_\infty(1)\neq 0$.
		We note that for every function $f\in L^1(\RR^d)$, by Fubini's theorem,
		\begin{equation}\label{eqn:fourierX}
			\widetilde  W_t(f)=\frac1{(2 \pi)^d}\int_{\RR^d} \widetilde W_t(\e_\theta)\hat f(\theta)d \theta\,.
		\end{equation}
		Hence,
		\begin{align*}
			\widetilde W_{\rho(t)}(\partial^kp_{t- \rho(t)})
			&=\frac1{(2 \pi)^d}\int_{\RR^d}e^{-(t- \rho(t))|\theta|^\alpha} \widetilde W_{\rho(t)}(\e_\theta)(i \theta)^{ k} d \theta\,.
		\end{align*}
		In addition, from \eqref{eqn:tildeWM}, we obtain
		\begin{equation*}
			\widetilde  W_{\rho(t)}(\e_{\theta})=m(\e_{\theta})-|\theta|^\alpha\int_0^{\rho(t)}\widetilde  W_s(\e_{\theta})ds+\widetilde M_{\rho(t)}(\e_{\theta})\,.
		\end{equation*}
		It follows that
		\begin{equation}\label{tmp:xe}
			\widetilde W_{\rho(t)}(\partial^kp_{t- \rho(t)})=I_1+I_2+I_3+I_4\,,
		\end{equation}
		where
		\begin{align*}
			I_1&=\frac1{(2 \pi)^d}\int_{\RR^d}e^{-(t- \rho(t))|\theta|^\alpha}m(\e_\theta)(i \theta)^{ k} d \theta\,,
			\\I_2&=-\frac1{(2 \pi)^d}\int_{\RR^d}e^{-(t- \rho(t))|\theta|^\alpha}\int_0^{\rho(t)}\widetilde W_s(\e_{\theta})ds |\theta|^\alpha (i \theta)^{ k} d \theta\,,
			\\I_3&=\frac1{(2 \pi)^d}\int_{\RR^d}e^{-(t- \rho(t))|\theta|^\alpha}\lt[\widetilde M_{\rho(t)}(\e_\theta)-\widetilde M_{\rho(t)}(1)\rt](i \theta)^{ k} d \theta\,,
			\\I_4&=\frac1{(2 \pi)^d}\int_{\RR^d}e^{-(t- \rho(t))|\theta|^\alpha}(i \theta)^{ k} d \theta\widetilde M_{\rho(t)}(1)\,.
		\end{align*}
		Applying Lemma \ref{lem:Mlim}, we see that
		\begin{equation}\label{tmp:I4}
			\lim_{t\to\infty}t^{\frac{d+|k|}\alpha}I_4=\frac1{(2 \pi)^d}\int_{\RR^d}e^{-|\theta|^\alpha}(i \theta)^{ k} d \theta\widetilde M_{\infty}(1)\,.
		\end{equation}
		We will show that
		\begin{align}
			\lim_{t\to\infty}t^{\frac{d+|k|}\alpha} I_1&=\frac1{(2 \pi)^d}\int_{\RR^d}e^{-|\theta|^\alpha}(i \theta)^k d \theta m(1)\,,\label{tmp:I1}
			\\\lim_{t\to\infty}t^{\frac{d+|k|}\alpha} I_2&=0
			\quad\textrm{and}\quad
			\lim_{t\to\infty}t^{\frac{d+|k|}\alpha} I_3=0\,.\label{tmp:I23}
		\end{align}
		By a change of variable, we have
		\begin{equation*}
			I_1=t^{-\frac{d+|k|}\alpha} \frac1{(2 \pi)^d}\int_{\RR^d}e^{-(1- \frac{\rho(t)}t)|\theta|^\alpha}m(\e_{t^{-1/\alpha} \theta})(i \theta)^{ k} d \theta  \,.
		\end{equation*}
		This, together with dominated convergence theorem yields \eqref{tmp:I1}. For $I_2$, we observe that
		\begin{equation*}
			\lt|\int_0^{\rho(t)}\widetilde W_s(\e_\theta)ds \rt|\le\int_0^{\rho(t)}\widetilde W_s(1)ds\lesssim \widetilde W_\infty(1) \rho(t)\,.
		\end{equation*}
		Hence,
		\begin{align*}
			|I_2|&\lesssim \widetilde W_\infty(1) \rho(t)\int_{\RR^d}e^{-(t- \rho(t))|\theta|^\alpha} |\theta|^{|k|+\alpha} d \theta
			\\&\lesssim \widetilde W_\infty(1)\frac{\rho(t)}tt^{-\frac{d+|k|}\alpha}\int_{\RR^d}e^{-(1- \frac{\rho(t)}t)|\theta|^\alpha} |\theta|^{|k|+\alpha} d \theta\,,
		\end{align*}
		which due to sublinearity of $\rho$ immediately implies the first assertion in \eqref{tmp:I23}.
		For $I_3$, putting $a_n=e^n$ and utilizing the Borel-Cantelli lemma, we merely need to show
		\begin{equation}\label{tmp:estI3}
			\sum_{n\ge1}\PP_m\lt(\sup_{a_n\le t\le a_{n+1}} t^{\frac{d+|k|}\alpha}|I_3|\rt)^2<\infty\,.
		\end{equation}
		Set $\rho_n=\rho(a_n)$ and note that
		\begin{equation*}
			\int_{\RR^d}e^{-(a_n- \rho_{n+1})|\theta|^\alpha}|\theta|^{|k|}d \theta
			\lesssim a_n^{-\frac{d+|k|}\alpha}\,.
		\end{equation*} 
		By Jensen's inequality, we have
		\begin{multline*}
			\PP_m\lt(\sup_{a_n\le t\le a_{n+1}} t^{\frac{d+|k|}\alpha}|I_3|\rt)^2
			\\\lesssim \lt(\frac{a_{n+1}^2}{a_n}\rt)^{\frac{d+|k|}\alpha} \int_{\RR^d}e^{-(a_n- \rho_{n+1})|\theta|^\alpha}\PP_m\lt[\sup_{t\ge0}|\widetilde M_t(\e_{\theta})-\widetilde M_t(1) |^2\rt]|\theta|^{|k|}d \theta.
		\end{multline*}
		Applying Lemma \ref{lem:supM}, we see that
		\begin{multline*}
			\PP_m\lt(\sup_{a_n\le t\le a_{n+1}}t^{\frac{d+|k|}\alpha}|I_3|\rt)^2
			\lesssim \lt(\frac{a_{n+1}^2}{a_n}\rt)^{\frac{d+|k|}\alpha} \int_{\RR^d}e^{-(a_n- \rho_{n+1})|\theta|^\alpha}(|\theta|^{2\wedge a}+ |\theta|^\alpha) |\theta|^{|k|}d \theta
			\\\lesssim \lt(\frac{a_{n+1}}{a_n}\rt)^{\frac{d+|k|}\alpha} \int_{\RR^d}e^{-(\frac{a_n}{a_{n+1}}- \frac{\rho_{n+1}}{a_{n+1}})|\theta|^\alpha}(a_{n+1}^{-(2\wedge a)/\alpha} |\theta|^{2\wedge a}+ a_{n+1}^{-1}|\theta|^\alpha) |\theta|^{|k|}d \theta\,.
		\end{multline*}
		Observing that $\frac{a_{n+1}}{a_n}=e$, $\lim_n(\frac{a_n}{a_{n+1}}- \frac{\rho_{n+1}}{a_{n+1}})=e^{-1}$ and $\sum_n a_n^{-\varepsilon}<\infty $ for any $\varepsilon>0$, the above estimate implies \eqref{tmp:estI3}.

		Finally, combining \eqref{tmp:xe}, \eqref{tmp:I4}, \eqref{tmp:I1} and \eqref{tmp:I23} yields
		\begin{equation*}
			\lim_{t\to\infty}t^{\frac{d+|k|}\alpha} W_{\rho(t)}(\partial^kp_{t- \rho(t)})=\frac{i^{|k|} }{(2 \pi)^d}\int_{\RR^d}e^{-|\theta|^\alpha}\theta^{ k}d \theta\lt(m(1)+\widetilde M_\infty(1) \rt) \,.
		\end{equation*}
		The equality \eqref{stb:pk} follows from the above relation and \eqref{eqn:MW}, after observing that $X_{\rho(t)}(\partial^kp_{t- \rho(t)})$ is a real number.
	\end{proof}
\section{Proof of the main result} 
\label{sec:proof_of_the_main_result}	
	We begin with an asymptotic expansion of $T_t$ as $t\to\infty$.
	If $k=(k_1,\dots,k_d)\in\NN^d$ is a multi-index and $f$ is a sufficiently smooth test function, we define $\partial^kf=\partial_{1}^{k_1}\partial_2^{k_1}\cdots\partial_d^{k_d}f$. The following semigroup expansion is proved in \cite{kl2017FV}*{Proposition 3.2}. 
	\begin{proposition}[Semigroup expansion]\label{prop:semiexpan}
		Let $f$ be a measurable function on $\RR^d$ and $N$ be a non-negative integer such that \eqref{con:fxN} holds.
		Then, we have
		\begin{align}\label{stb:Pexpansion}
			\lim_{t\to\infty}t^{\frac {N+d}\alpha}\sup_{x\in\RR^d} \lt|T_tf(x)-\sum_{k\in\NN^d:|k|\le N}\frac{(-1)^{|k|}}{k!}\int_{\RR^d}f(y)y^{k}dy\partial^{k}p_t(x)\rt|=0\,.
		\end{align}
	\end{proposition}
\begin{remark}
	For semigroups with discrete spectra such as the Ornstein-Uhlenbeck semigroup, similar asymptotic expansions can be obtained via spectral decompositions. Although the $\alpha$-stable semigroup does not belong to this class, such expansion can be obtained using Taylor's expansion. We refer to \cite{kl2017FV} for a proof of the above result.
\end{remark}
Set $t_n=n^\delta$ for some $\delta\in(0,1)$ sufficiently small so that
\begin{equation}\label{d.range}
	\delta\frac{N+d}\alpha+\delta<1\,.
\end{equation}
We first show that the sequence $\{t_n\}_n$ determines the long-time asymptotic of $\widetilde W_t$.
\begin{lemma}\label{lem:tndetermine}
	For every $f\in\cdd(-(-\Delta)^{\frac \alpha2})$, we have
	\begin{equation}
		\lim_n \sup_{t\in[t_n,t_{n+1}]}\lt|t^{\frac{d+N}\alpha}\widetilde W_t(f)-t_n^{\frac{d+N}\alpha}\widetilde W_{t_n}(f)\rt|=0 \quad\PP_m\textrm{-a.s.}
	\end{equation}
\end{lemma}
\begin{proof}
	We observe that
	\begin{align*}
		\sup_{t\in[t_n,t_{n+1}]} |t^{\frac{d+N}\alpha}\widetilde W_t(f)&-t_n^{\frac{d+N}\alpha}\widetilde W_{t_n}(f)|\le J_1+J_2+J_3\,,
	\end{align*}
	where
	\begin{align*}
		J_1&=\sup_{t\in[t_n,t_{n+1}]}t^{\frac{N+d}\alpha}|\widetilde W_{t}(f)-\widetilde W_t(T_{t_{n+1}-t}f)|\,,
		\\J_2&=\sup_{t\in[t_n,t_{n+1}]}t^{\frac{N+d}\alpha}|\widetilde W_t(T_{t_{n+1}-t}f)-\widetilde W_{t_n}(T_{t_{n+1}-t_n}f)|\,,
		\\J_3&=\sup_{t\in[t_n,t_{n+1}]}|t^{\frac{N+d}\alpha}\widetilde W_{t_n}(T_{t_{n+1}-t_n}f)-t_n^{\frac{N+d}\alpha}\widetilde W_{t_n}(f)|\,.
	\end{align*}
	Hence, it suffices to show $\lim_nJ_1=\lim_nJ_2=\lim_nJ_3=0$ almost surely.
	Indeed, we have
	\begin{align*}
		J_1\le \sup_{t\in[t_n,t_{n+1}]}t^{\frac{N+d}\alpha}\|T_{t_{n+1}-t}f-f\|_\infty \sup_{t\in[t_n,t_{n+1}]}\widetilde W_t(1)\,.
	\end{align*}
	Since $f\in\cdd(-(-\Delta)^{\frac \alpha2})$, we see that
	\begin{equation}\label{tmp:Tn}
		\sup_{t\in[t_n,t_{n+1}]}t^{\frac{N+d}\alpha}\|T_{t_{n+1}-t}f-f\|_\infty\lesssim t_{n+1}^{\frac{N+d}\alpha}|t_{n+1}-t_n|\,,
	\end{equation}
	which converges to 0 because of the range of $\delta$ in \eqref{d.range}.
	In addition, 
	\begin{equation}\label{tmp:linfty}
	 	\lim_n\sup_{t\in[t_n,t_{n+1}]}\widetilde W_t(1)\le\limsup_{t\to\infty}\widetilde W_t(1)= \widetilde W_\infty(1)\,.
	\end{equation} 
	Hence, $\lim_nJ_1=0$ almost surely. For $J_2$, we observe from \eqref{eqn:WW} that for every $t\in[t_n,t_{n+1}]$,
	\begin{align*}
		\widetilde W_t(T_{t_{n+1}-t}f)-\widetilde W_{t_n}(T_{t_{n+1}-t_n}f)=\int_{t_n}^t\int_{\RR^d}T_{t_{n+1}-r}f(y)d\widetilde M(r,y)\,.
	\end{align*}
	Fixing $\varepsilon>0$ and applying martingale maximal inequality as well as \eqref{eqn:quadM} and \eqref{eqn:1stXmoment}, we have
	\begin{align*}
		\PP_m(J_2>\varepsilon)&\le\PP_m\lt(t_{n+1}^{\frac{N+d}\alpha}\sup_{t\in[t_n,t_{n+1}]}\lt|\int_{t_n}^t\int_{\RR^d}T_{t_{n+1}-r}f(y)d\widetilde M(r,y) \rt|>\varepsilon \rt)
		\\&\le \varepsilon^{-2}t_{n+1}^{2\frac{N+d}\alpha}\PP_m\lt|\int_{t_n}^{t_{n+1}}\int_{\RR^d}T_{t_{n+1}-r}f(y)d\widetilde M(r,y) \rt|^2
		\\&\le \varepsilon^{-2}t_{n+1}^{2\frac{N+d}\alpha}\int_{t_n}^{t_{n+1}}\langle T_r \lt((T_{t_{n+1}-r}f)^2\rt),m \rangle e^{-\beta r}dr\,.
	\end{align*}
	As in the proof of Lemma \ref{lem:WW}, an application of Jensen's inequality gives 
	\begin{equation*}
		\langle T_r \lt((T_{t_{n+1}-r}f)^2\rt),m \rangle\le \langle T_t(f^2) ,m \rangle\le m(1)\|f\|_\infty^2\,.
	\end{equation*}
	It follows that $\PP_m(J_2>\varepsilon)\lesssim \varepsilon^{-2}t_{n+1}^{2\frac{N+d}\alpha}e^{-\beta t_n}$ and, hence $\sum_n\PP_m(J_2>\varepsilon)<\infty$.
	Applying Borel-Cantelli lemma, we find that $\lim_n J_2=0$ almost surely. $J_3$ can be treated analogously as $J_1$. Indeed, we have
	\begin{align*}
		J_3\le\sup_{t\in[t_n,t_{n+1}]}\|t^{\frac{N+d}\alpha}T_{t_{n+1}-t_n}f-t_n^{\frac{N+d}\alpha}f\|_\infty\sup_{t\in[t_n,t_{n+1}]}\widetilde W_t(1)\,.
	\end{align*}
	By triangle inequality, we see that $\sup_{t\in[t_n,t_{n+1}]}\|t^{\frac{N+d}\alpha}T_{t_{n+1}-t_n}f-t_n^{\frac{N+d}\alpha}f\|_\infty$ is at most
	\begin{align*}
		\sup_{t\in[t_n,t_{n+1}]}t^{\frac{N+d}\alpha}\|T_{t_{n+1}-t_n}f-f\|_\infty
		+(t^{\frac{N+d}\alpha}_{n+1}-t^{\frac{N+d}\alpha}_n)\|f\|_\infty\,,
	\end{align*}
	which converges to 0 by \eqref{tmp:Tn} and \eqref{d.range}. In conjunction with \eqref{tmp:linfty}, these estimates imply that $\lim_nJ_3=0$ almost surely.
\end{proof}
\begin{proof}[Proof of Theorem \ref{thm:stbSLLn}]
	We are going to obtain the limit \eqref{stb1} along the sequence $\{t_n\}_n$. We put $\rho(t)=\sqrt t$ and
	\begin{equation}
		L_tf=\sum_{k\in\NN^d:|k|\le N}(-1)^{|k|}\lambda^k_d(f)\partial^{k}p_t \quad\forall t\ge0\,.
	\end{equation} 
	We will show that
	\begin{equation}\label{tmp:convtn}
		\lim_n t_n^{\frac{N+d}\alpha}|\widetilde W_{t_n}(f)-\widetilde W_{\rho(t_n)}(L_{t_n- \rho(t_n)}f)|=0 \quad \PP_m\textrm{-a.s.}
	\end{equation}
	From Lemma \ref{lem:WW}, we see that
	\begin{equation*}
		\PP_m\lt[\lt(\widetilde W_{t_n}(f)-\widetilde W_{\rho(t_n)}(T_{t_n- \rho(t_n)}f) \rt)^2\rt]\lesssim e^{-\beta \rho(t_n)}\,.
	\end{equation*}
	This implies that
	\begin{equation*}
		\sum_n\PP_m\lt[t_n^{2\frac{N+d}\alpha} \lt(\widetilde W_{t_n}(f)-\widetilde W_{\rho(t_n)}(T_{t_n- \rho(t_n)}f) \rt)^2\rt]
		\lesssim \sum_n t_n^{2\frac {N+d} \alpha}e^{-\beta \rho(t_n)}<\infty\,.
	\end{equation*}
	An application of  Borel-Cantelli lemma yields
	\begin{equation}\label{tmp:WL}
		\lim_{n}t_n^{\frac{N+d}\alpha}|\widetilde W_{t_n}(f)-\widetilde W_{\rho(t_n)}(T_{t_n- \rho(t_n)}f)|=0 \quad\PP_m\textrm{-a.s.}
	\end{equation}
	In addition,
	\begin{equation*}
		t_n^{\frac{N+d}\alpha}|\widetilde W_{\rho(t_n)}(T_{t_n- \rho(t_n)}f)-\widetilde W_{\rho(t_n)}(L_{t_n- \rho(t_n)}f)|
		\le t_n^{\frac{N+d}\alpha}\|T_{t_n- \rho(t_n)}f-L_{t_n- \rho(t_n)}f\|_\infty \widetilde W_{\rho(t_n)}(1)\,.
	\end{equation*}
	We note that $\lim_n\widetilde W_{\rho(t_n)}(1)=\widetilde W_\infty(1)$ and $\lim_nt_n^{\frac{N+d}\alpha}\|T_{t_n- \rho(t_n)}f-L_{t_n- \rho(t_n)}f\|_\infty=0$ (by \eqref{stb:Pexpansion}). It follows that
	\begin{equation*}
		\lim_nt_n^{\frac{N+d}\alpha}|\widetilde W_{\rho(t_n)}(T_{t_n- \rho(t_n)}f)-\widetilde W_{\rho(t_n)}(L_{t_n- \rho(t_n)}f)|=0 \quad\PP_m\textrm{-a.s.} \,,
	\end{equation*}
	which together with \eqref{tmp:WL} implies \eqref{tmp:convtn}. More precisely, we have shown
	\begin{equation*}
		\lim_n t_n^{\frac{N+d}\alpha}\lt|\widetilde W_{t_n}(f)-\sum_{k\in\NN^d:|k|\le N}(-1)^{|k|}\lambda^k_d(f)\widetilde W_{\rho(t_n)}(\partial^{k}p_{t_n- \rho(t_n)}) \rt|=0 \quad\PP_m\textrm{-a.s.}
	\end{equation*}
	Applying Proposition \ref{prop:Wp}, we see that for every $k\in\NN^d$,
	\begin{equation*}
			\lim_{n}t^{\frac{d+|k|}\alpha}_n \widetilde W_{\rho(t_n)}(\partial^kp_{t_n- \rho(t_n)})=
			\lt\{
			\begin{array}{ll}
				0& \quad\mbox{if } |k|\mbox{ is odd}\\
				(-1)^{\frac {|k|}2} \vartheta^k_{d,\alpha}\widetilde W_\infty(1) & \quad\mbox{if } |k|\mbox{ is even}\,.
			\end{array}
			\rt.
	\end{equation*} 
	Combining these limits together yields
	\begin{equation*}
		\lim_{n}t^{\frac {N+d} \alpha}_n\lt|W_{t_n}(f)-\widetilde W_\infty(1)\sum_{\substack{k\in\NN^d:|k|\le N\\|k|\textrm{ is even}}} {(-1)^{\frac{|k|}2}t_n^{-\frac{d+|k|}\alpha}}\vartheta^k_{d,\alpha}\lambda^k_d(f)\rt|=0\,.
	\end{equation*}
	Applying Lemma \ref{lem:tndetermine}, we find that the above limit implies \eqref{stb1}.
\end{proof}

\bibliography{superprocesses}

\end{document}